\documentclass[11pt]{amsart}
\usepackage{amssymb}

\newtheorem{theorem}{Theorem}[section]

\newcommand{\adfmod}[1]{~(\mathrm{mod}~#1)}    
\newcommand{\adfPENT}{\mathop{\mathrm{PENT}} } 
\newcommand{\ADFvfyParStart}[1]{}
\newcommand{\adfsplit}{\par}                   
\newcommand{\adfhide}[1]{}                     
\begin{document}
\title{Existence results for pentagonal geometries}
\author{A. D. Forbes}
\address{School of Mathematics and Statistics\\
        The Open University\\
        Walton Hall, Milton Keynes MK7 6AA, UK}
\email{anthony.d.forbes@gmail.com}
\author{T. S. Griggs}
\address{School of Mathematics and Statistics\\
        The Open University\\
        Walton Hall, Milton Keynes MK7 6AA, UK}
\email{terry.griggs@open.ac.uk}
\author{K. Stokes}
\address{Department of Mathematics and Mathematical Statistics\\
        Ume\aa University\\
        901 87 Ume\aa, SWEDEN}
\email{klara.stokes@umu.se}


\maketitle
\begin{abstract}
%
New results on pentagonal geometries PENT($k,r$) with block sizes $k = 3$ or $k = 4$ are given. In particular we completely determine the existence spectra for PENT($3,r$) systems with the maximum number of opposite line pairs as well as those without any opposite line pairs. A wide-ranging result about PENT($3,r$) with any number of opposite line pairs is proved. We also determine the existence spectrum of PENT($4,r$) systems with eleven possible exceptions.
\end{abstract}

\section{Introduction}\label{sec:Introduction}
Generalized polygons were introduced by Tits \cite{T} over sixty years ago and can be described as follows. A \textit{partial linear space} is an ordered pair $(V,\mathcal{L})$ where $V$ is a set of elements, usually called \textit{points}, of cardinality $v$ and $\mathcal{L}$ is a family of subsets of $V$, usually called \textit{lines} or \textit{blocks}, such that every pair of distinct points is contained in at most one line. The number of lines is denoted by $b$. If every line has the same cardinality $k \geq 2$, the space is said to be \textit{uniform} and if every point is incident with the same number $r \geq 1$ of lines it is said to be \textit{regular}.

Denote such a uniform regular partial linear space by PLS($k,r$). A \textit{generalized polygon} is a PLS($k,r$) for which the girth of the \textit{point-line incidence graph} or \textit{Levi graph} is twice the diameter, $n$. The case where $k=r=2$ are ordinary polygons. In \cite{FH}, Feit \& Higman proved that the only finite examples are \textit{thin} (two points on each line or two lines through each point) or $n \in \{2,3,4,6,8\}$. So there are no (thick and finite) generalized pentagons.

This result motivated the authors of \cite{BBDS} to introduce an alternative way to generalize the pentagon. A \textit{pentagonal geometry} PENT($k,r$) is a partial linear space PLS($k,r$) in which for all points $x \in V$, the points not collinear with $x$ are themselves collinear. We call this line the \textit{opposite line} to $x$ and denote it by $x^{\mathrm{opp}}$. If two points $x$ and $y$ have the same opposite line $x^{\mathrm{opp}}=y^{\mathrm{opp}}=l$, then $z^{\mathrm{opp}}=m$ for all points $z \in l$ where $m$ is the line joining $x$ and $y$. Similarly $w^{\mathrm{opp}}=l$ for all $w \in m$. Such a pair of lines $(l,m)$ is called an \textit{opposite line pair}. The pentagon is the geometry PENT($2,2$) and the Desargues configuration is PENT($3,3$). When $r=1$, PENT($k,1$) consists of two  disjoint lines, each of cardinality $k$. This is a \textit{degenerate} pentagonal geometry.

The basic theory of pentagonal geometries was also developed in \cite{BBDS}. The following results are of fundamental importance.
\begin{theorem}\label{count}
A pentagonal geometry $\adfPENT(k,r)$ has $rk-r+k+1$ points and $(rk-r+k+1)r/k$ lines. Thus a necessary condition for existence is that $k$ divides $r(r-1)$.
\end{theorem}
\begin{theorem}\label{ineq}
If there exists a pentagonal geometry $\adfPENT(k,r)$ with $r > 1$, then $r \geq k$.
\end{theorem}
\begin{theorem}\label{opplinepair}
A pentagonal geometry $\adfPENT(k,r)$ with $1<r<3k$ has either\\
(i) no opposite line pair, or\\
(ii) $r=2k+1$ and the points are partitioned into opposite line pairs.
\end{theorem}

An important concept in the theory of partial linear spaces is that of the \textit{leave} or \textit{deficiency graph}. This is the graph $G$ whose vertex set is $V$ with two points $x$ and $y$ being adjacent if and only if they are not collinear. The following result is also proved in \cite{BBDS}.
\begin{theorem}\label{leave}
The deficiency graph $G$ of a pentagonal geometry $\adfPENT(k,r)$ is the disjoint union of complete bipartite graphs $K_{k,k}$ (one for each opposite line pair) and $G'$ where $G'$ is a $k$-regular graph of girth at least $5$, not necessarily connected.
\end{theorem}

Turning now to existence results, the authors of \cite{BBDS} used the previous lemma to relate the existence of a pentagonal geometry PENT($k,k$) or PENT($k,k-1$) to that of a \textit{Moore graph} of girth 5, i.e.\ a $k$-regular graph with $k^2+1$ vertices. They proved the two following theorems.
\begin{theorem}\label{kk}
A pentagonal geometry $\adfPENT(k,k)$ exists only for $k = 2,3,7$ and possibly $57$.
\end{theorem}
\begin{theorem}\label{kplus1}
A pentagonal geometry $\adfPENT(k,k+1)$ exists only for $k = 2,6$ and possibly $56$.
\end{theorem}

In addition, the case where $k=2$ can be completely solved. From Theorems \ref{count} and \ref{leave} we have the following theorem, which is also taken from \cite{BBDS}.
\begin{theorem}\label{rplus3}
A pentagonal geometry $\adfPENT(2,r)$ is a complete graph on $r+3$ vertices from which a union of disjoint cycles, none of size $3$, spanning the vertex set has been deleted.
\end{theorem}

This allows the number of non-isomorphic pentagonal geometries PENT($2,r$) to be determined. Let $p(n)$ be the \textit{partition function}. The following theorem was proved in \cite{GS}.
\begin{theorem}\label{partition}
The number of non-isomorphic pentagonal geometries $\adfPENT(2,r)$ is
$p(r+3)-p(r+2)-p(r+1)+p(r-1)+p(r-2)-p(r-3)$.
\end{theorem}

This paper is mainly concerned with existence results for pentagonal geometries with block size 3 or 4. In Section \ref{sec:Block size 3} we completely determine the existence spectra for pentagonal geometries PENT($3,r$) with the maximum number of opposite line pairs as well as those without any opposite line pairs. In the former case, we also present an implementation of the construction which gives pentagonal geometries PENT($3,9s-2$) and PENT($3,9s+1$), $s \geq 1$, directly from Steiner triple systems using a method which goes back to Bose \cite{B}. The method also extends to PENT($3,9s+4$), $s \geq 1$, by using a pairwise balanced design. Section \ref{sec:Small systems} is devoted to our computer calculations for ``small'' pentagonal geometries PENT($3,r$), $1 \leq r \leq 12$. Using the PENT($3,10$) with one opposite line pair referred to in that section and listed explicitly in Appendix B we are able to prove a wide-ranging result about PENT($3,r$) with any number of opposite lie pairs.

Section \ref{sec:Block size 4} deals with block size 4. The PENT($4,13$), PENT($4,20$) and PENT($4,24$) in that section, as well as the systems in Appendix C, are the first examples of pentagonal geometries with $k \geq 4$ with a connected deficiency graph. In Theorems \ref{pent4,1} and \ref{pent4,0} we determine the existence spectrum for pentagonal geometries PENT($4,r$) with just eleven possible exceptions which is a considerable advance on the previous results on these geometries by two of the present authors \cite{GS}.  

Finally in this section we recall two constructions from \cite{GS} which are powerful tools in the construction of pentagonal geometries. First we need a definition. A \textit{$k$-group divisible design}, $k$-GDD, is an ordered triple ($V$,$\mathcal{G}$,$\mathcal{B}$), where $V$ is a set of \textit{points} of cardinality $v$, $\mathcal{G}$ is a partition of $V$ into \textit{groups} and $\mathcal{B}$ is a family of subsets of $V$, called \textit{lines} or \textit{blocks}, each of cardinality $k$, such that every pair of distinct points is contained in either precisely one group or one block, but not both. If $v= a_1g_1 + a_2g_2 + \ldots + a_sg_s$ and if there are $a_i$ groups of cardinality $g_i$, $i = 1,2,\ldots,s$, then the $k$-GDD is said to be of \textit{type} $g_1^{a_1}g_2^{a_2} \ldots g_s^{a_s}$. A $k$-GDD in which every group has the same cardinality is said to be \textit{uniform}.
\begin{theorem}\label{g^u}
Let $\adfPENT(k,r)$ be a (possibly degenerate) pentagonal geometry. If there exists a $k$-$\mathrm{GDD}$ of type $((k-1)r+(k+1))^u$, then there exists a pentagonal geometry $\adfPENT(k,ru+(k+1)(u-1)/(k-1))$.
\end{theorem}
\begin{theorem}\label{g^um^1}
Let $\adfPENT(k,r)$ and $\adfPENT(k,s)$ be (possibly degenerate) pentagonal geometries. If there exists a $k$-$\mathrm{GDD}$ of type $((k-1)r+(k+1))^u((k-1)s+(k+1))^1$, then there exists a pentagonal geometry $\adfPENT(k,(r+(k+1)/(k-1))u+s)$.
\end{theorem}
Theorem \ref{g^u} is of course a special case of Theorem \ref{g^um^1} and indeed both are special cases of the next more general theorem. 
\begin{theorem}\label{gddgeneral}
Let $k \ge 2$ be an integer. For $i = 1$, $2$, \dots, $n$, let $r_i$ be a positive integer, let $v_i = (k-1) r_i+(k+1)$
and suppose there exists a pentagonal geometry $\adfPENT(k, r_i)$.
Suppose also that there exists a $k$-$\mathrm{GDD}$ of type $v_1^1 v_2^1 \dots v_n^1$. Then there exists a  pentagonal geometry $\adfPENT(k,r_1+r_2 +\dots+r_n+(n-1)(k+1)/(k-1))$.
\end{theorem}
\begin{proof}
On each group of points of cardinality $v_i$ of the GDD, construct a pentagonal geometry PENT($k,r_i$), $i=1,2,\ldots,n$ and adjoin to the blocks of the GDD.
\end{proof}


\section{Block size 3}\label{sec:Block size 3}
From Theorem \ref{count}, a necessary condition for the existence of a pentagonal geometry PENT($3,r$) is $r\equiv$~0~or~1~(mod 3). An elementary calculation shows that an upper bound for the number of opposite line pairs in such a system is $v/6=(r+2)/3$ if $r\equiv$~1~(mod 3) and $(v-10)/6=(r-3)/3$ if $r\equiv$~0~(mod 3). We will call such systems \textit{opposite line pair maximal} or simply just \textit{maximal}. Note that for $r\equiv$~1~ (mod 3) the point set of an opposite line pair maximal pentagonal geometry is partitioned into opposite line pairs. For $r\equiv$~0~(mod 3) the deficiency graph is a disjoint union of complete bipartite graphs $K_{3,3}$ and the Petersen graph. The 10 vertices of the latter will form a Desargues configuration (PENT($3,3$)) in the pentagonal geometry and the other points will be partitioned into opposite line pairs. It will be convenient to recall the following result from \cite{GS} including the proof, which is quite short and will be useful as a reference to compare with the proof of Theorem \ref{pent3none}. Throughout this section, existence results for 3-GDDs of type $g^u$ come from \cite{H} and of type $g^um^1$ from \cite{CHR}, see also \cite{HBGe}.
\begin{theorem}\label{pent3max}
The existence spectrum of opposite line pair maximal pentagonal geometries $\adfPENT(3,r)$ is $r\equiv 0$~or~$1\pmod{3}$, except for $r \in \{4,6,9\}$.
\end{theorem}
\begin{proof}
In Theorem \ref{g^u}, let $k=3$ and $r=1$. There exists a pentagonal geometry PENT($3,1$) and a 3-GDD of type $6^t$, $t \geq 3$. Hence there exists a pentagonal geometry PENT($3,3t-2$), $t \geq 3$. We have already observed that there exists a PENT($3,1$) and from Theorem \ref{kplus1} there is no PENT($3,4$). From the construction it is clear that these systems are maximal.

In Theorem \ref{g^um^1}, let $k=3$, $r=1$ and $s=3$. There exist pentagonal geometries PENT($3,1$) and PENT($3,3$) and a 3-GDD of type $6^t10^1$, $t \geq 3$. Hence there exists a pentagonal geometry PENT($3,3t+3$), $t \geq 3$. Again we have already observed that a PENT($3,3$) exists and it was shown in \cite{BBDS} that there is no PENT($3,6$). Again it is clear from the construction that these systems are maximal. This just leaves the case where $r=9$. A maximal system would contain two opposite line pairs but this was shown to be impossible in \cite{GS}.
\end{proof}

However we note that a pentagonal geometry PENT($3,9$) with one opposite line pair does exist and was given in \cite{BBDS}. For completeness we include it here. The number of points $v$ is 22 and the number of lines $b$ is 66.\\~\\

\newpage
\noindent\textbf{PENT(\boldmath$3,9$\unboldmath) with one opposite line pair.}\\
{\small\begin{tabular}{llllll}
\{0,1,2\},&\{0,7,12\},&\{0,16,19\},&\{0,9,13\},&\{0,17,18\},&\{0,10,21\},\\
\{0,6,20\},&\{0,11,15\},&\{0,8,14\},&\{1,10,12\},&\{1,16,18\},& \{1,11,21\},\\
\{1,7,20\},&\{1,6,19\},&\{1,8,17\},&\{1,9,15\},&\{1,13,14\},&\{2,11,12\},\\
\{2,15,16\},&\{2,8,21\},&\{2,9,20\},&\{2,13,19\},&\{2,7,18\},&\{2,6,17\},\\
\{2,10,14\},&\{3,4,5\},&\{3,12,14\},&\{3,6,16\},&\{3,13,21\},&\{3,8,20\},\\
\{3,10,19\},&\{3,11,18\},&\{3,9,17\},&\{3,7,15\},&\{4,12,16\},& \{4,14,20\},\\
\{4,9,10\},&\{4,7,21\},&\{4,8,19\},&\{4,13,18\},&\{4,11,17\},&\{4,6,15\},\\
\{5,12,20\},&\{5,9,16\},&\{5,14,17\},&\{5,8,11\},&\{5,6,21\},&\{5,7,19\},\\
\{5,10,18\},&\{5,13,15\},&\{6,10,11\},&\{6,12,13\},&\{6,14,18\},& \{7,8,9\},\\
\{7,13,17\},&\{7,14,16\},&\{8,15,18\},&\{8,10,16\},&\{9,11,19\},& \{9,12,21\},\\
\{10,15,20\},&\{11,13,20\},&\{12,17,19\},&\{14,15,21\},&\{16,17,21\},& \{18,19,20\}.
\end{tabular}}\\~\\
The opposite line pair is $\{0,1,2\}, \{3,4,5\}$.\\~\\
There also exists a PENT($3,9$) with no opposite line pair, which will be needed in the proof of the next theorem. \\~\\
\noindent\textbf{PENT(\boldmath$3,9$\unboldmath) with no opposite line pair.}\\
{\small\begin{tabular}{llllll}
\{0,3,4\},&\{0,5,12\},&\{0,6,14\},&\{0,7,17\},&\{0,8,16\},&\{0,9,11\},\\
\{0,10,19\},&\{0,13,15\},&\{0,18,21\},&\{1,2,20\},&\{1,3,12\},&\{1,4,9\},\\
\{1,5,14\},&\{1,6,13\},&\{1,8,18\},&\{1,10,15\},&\{1,11,16\},&\{1,19,21\},\\
\{2,5,6\},&\{2,7,15\},&\{2,8,14\},&\{2,9,19\},&\{2,10,18\},&\{2,11,13\},\\
\{2,12,17\},&\{2,16,21\},&\{3,5,7\},&\{3,6,16\},&\{3,8,15\},&\{3,10,20\},\\
\{3,11,14\},&\{3,13,18\},&\{3,17,21\},&\{4,7,8\},&\{4,10,16\},& \{4,11,21\},\\
\{4,12,18\},&\{4,13,20\},&\{4,14,19\},&\{4,15,17\},&\{5,8,19\},& \{5,9,18\},\\
\{5,10,17\},&\{5,13,16\},&\{5,15,20\},&\{6,9,10\},&\{6,11,20\},& \{6,12,21\},\\
\{6,15,18\},&\{6,17,19\},&\{7,9,16\},&\{7,10,21\},&\{7,11,18\},& \{7,12,19\},\\
\{7,14,20\},&\{8,11,12\},&\{8,13,21\},&\{8,17,20\},&\{9,12,20\},& \{9,13,17\},\\
\{9,14,21\},&\{10,13,14\},&\{11,15,19\},&\{12,15,16\},&\{14,17,18\},& \{16,19,20\}.
\end{tabular}}\\~\\
Probably of more interest though are pentagonal geometries PENT($3,r$) with no opposite line pairs. We are still able to use Theorems \ref{g^u} and \ref{g^um^1} to construct these but we must use as constituent geometries those which also have no opposite line pairs. In particular therefore we cannot use PENT($3,1$) but can use  both PENT($3,3$) and the PENT($3,9$) with no opposite line pair given above. The three systems PENT($3,r$) for $r \in \{12, 15, 21\}$ given below also fall into this category; as is easily verified their deficiency graphs are connected. By using these systems we can determine the existence spectrum apart from a small number of possible exceptions which however can be dealt with by direct construction. All three systems were constructed by hand.\\

\noindent\textbf{PENT(\boldmath$3,12$\unboldmath).}\\
The number of points $v$ is $28=4 \times 7$ and the number of lines $b$ is $112=16 \times 7$. Assume an automorphism of order 7 and let the points be denoted by $(i,j)$
where $0 \leq i \leq 3$ and $0 \leq j \leq 6$. Let the automorphism be $(i,j) \mapsto (i,j+1)$ (mod 7). For convenience write $(0,j)$ as $\mathrm{a} j$, $(1,j)$ as $\mathrm{b} j$, $(2,j)$ as $\mathrm{c} j$ and $(3,j)$ as $\mathrm{d} j$. The edges of the deficiency graph are given by
$\mathrm{a} 0 \sim \{\mathrm{b} 1,\mathrm{c} 6,\mathrm{d} 0\}$; $\mathrm{b} 0 \sim \{\mathrm{a} 6,\mathrm{c} 1,\mathrm{d} 0\}$;
$\mathrm{c} 0 \sim \{\mathrm{a} 1,\mathrm{b} 6,\mathrm{d} 0$\}; $\mathrm{d} 0 \sim \{\mathrm{a} 0,\mathrm{b} 0,\mathrm{c} 0\}$
with other connections determined by the automorphism. There are 16 orbits under the automorphism, four of which are generated by the opposite lines above. Orbit starters for the 12 others are as follows.

\vskip 1mm
\begin{center}
\begin{tabular}{llll}
\{a1,a5,b0\},&\{a2,a4,c0\},&\{a2,a3,d0\},&\{b1,b4,c0\},\\
\{b2,b4,d0\},&\{b3,b4,a0\},&\{c3,c4,d0\},&\{c1,c4,a0\},\\
\{c2,c4,b0\},&\{d2,d3,a0\},&\{d2,d4,b0\},&\{d2,d5,c0\}.\\
\end{tabular}
\end{center}

\vskip 1mm
\noindent\textbf{PENT(\boldmath$3,15$\unboldmath).}\\
The number of points $v$ is $34=2 \times 17$ and the number of lines $b$ is $170=10 \times 17$. Assume an automorphism of order 17 and let the points be denoted by $(i,j)$
where $i \in \{0,1\}$ and $0 \leq j \leq 16$. Let the automorphism be $(i,j) \mapsto (i,j+1)$ (mod 17).
Write $(0,j)$ as $\mathrm{a} j$ and $(1,j)$ as $\mathrm{b} j$.
The edges of the deficiency graph are given by
$\mathrm{a} 0 \sim \{\mathrm{b} 0,\mathrm{a} 1,\mathrm{a} 16\}$ and $\mathrm{b} 0 \sim \{\mathrm{a} 0,\mathrm{b} 2,\mathrm{b} 15\}$
with other connections determined by the automorphism. There are 10 orbits under the automorphism, two of which are generated by the opposite lines above. Orbit starters for the 8 others are as follows.

\vskip 1mm
\begin{center}
\begin{tabular}{llll}
\{a0,a3,a8\},&\{a0,a4,b7\},&\{a0,a6,b14\},&\{a0,a7,b13\},\\
\{a7,b0,b1\},&\{a12,b0,b7\},&\{a13,b0,b5\},&\{b0,b3,b9\}.\\
\end{tabular}
\end{center}

\vskip 1mm
\noindent\textbf{PENT(\boldmath$3,21$\unboldmath).}\\
The number of points $v$ is $46=2 \times 23$ and the number of lines $b$ is $322=14 \times 23$. Assume an automorphism of order 23 and let the points be denoted by $(i,j)$
where $i \in \{0,1\}$ and $0 \leq j \leq 22$. Let the automorphism be $(i,j) \mapsto (i,j+1)$ (mod 23).
Again write $(0,j)$ as $\mathrm{a} j$ and $(1,j)$ as $\mathrm{b} j$.
The edges of the deficiency graph are given by
$\mathrm{a} 0 \sim \{\mathrm{b} 0,\mathrm{a} 1,\mathrm{a} 22$\} and $\mathrm{b} 0 \sim \{\mathrm{a} 0,\mathrm{b} 2,\mathrm{b} 21\}$
with other connections determined by the automorphism. There are 14 orbits under the automorphism, two of which are generated by the opposite lines above. Orbit starters for the 12 others are as follows.

\begin{center}
\begin{tabular}{llll}
\{a0,a3,a11\},&\{a0,a4,a9\},&\{a0,a6,a13\},&\{a20,b0,b1\},\\
\{a11,b0,b3\},&\{a12,b0,b5\},&\{a16,b0,b6\},&\{a13,b0,b7\},\\
\{a17,b0,b8\},&\{a4,b0,b9\},&\{a15,b0,b10\},&\{a14,b0,b11\}.\\
\end{tabular}
\end{center}

\begin{theorem}\label{pent3none}
The existence spectrum of pentagonal geometries $\adfPENT(3,r)$ with no opposite line pair is $r\equiv 0$~or~$1\pmod{3}$, except for $r \in \{1,4,6,7\}$.
\end{theorem}
\begin{proof}
The proof follows closely that of Theorem \ref{pent3max}.
In Theorem \ref{g^u}, let $k=3$ and $r=3$. There exists a pentagonal geometry PENT($3,3$) with no opposite line pair and a 3-GDD of type
$10^u$, $u=3t$ or $3t+1$, $t \geq 1$. Hence there exist pentagonal geometries PENT($3,15t-2$) and PENT($3,15t+3$), $t \geq 1$ with no opposite line pairs. This deals with the residue classes 3~and~13~(mod 15).

In Theorem \ref{g^um^1}, let $k=3$, $r=3$ and $s=9$. There exist pentagonal geometries PENT($3,3$) and PENT($3,9$) with no opposite line pairs and a 3-GDD of type $10^u22^1$, $u=3t$, $t \geq 2$ and $u=3t+2$, $t \geq 1$. Hence there exist pentagonal geometries PENT($3,15t+9$), $t \geq 2$ and PENT($3,15t+19$), $t \geq 1$. This deals with the residue classes 4~and~9~(mod 15) except for the values $r=4$ (which does not exist), 19 and 24.

For the residue classes (i) 7~and~12, (ii) 0~and~10, (iii) 1~and~6~(mod 15) proceed as above with (i) $s=12$ and a 3-GDD of type $10^u28^1$, (ii) $s=15$ and a 3-GDD of type $10^u34^1$, (iii) $s=21$ and a 3-GDD of type $10^u46^1$ instead of $s=9$ and a 3-GDD of type $10^u22^1$. The details are exactly the same except that in the latter case the range of values for PENT($3,15t+31$) is $t \geq 2$. The missing values are (i) $r=7$ (which does not exist), 22 and 27, (ii) $r=10$, 25 and 30, (iii) $r=1$ and 6 (neither of which exist), 16, 31, 36 and 46.

For $\adfPENT(3, r)$, $ r \in$ \{10, 16, 19, 22, 24, 25, 27\},
see Appendix A.

For $ r \in$ \{30, 31, 36, 46\}, use Theorem~\ref{gddgeneral} as follows:\\
$\adfPENT(3,30)$: 4 of $\adfPENT(3,3)$, $\adfPENT(3,10)$ and a 3-GDD of type $10^4 24^1$;\\
$\adfPENT(3, 31)$: 3 of $\adfPENT(3,9)$ and a 3-GDD of type $22^3$;\\
$\adfPENT(3, 36)$: 3 of $\adfPENT(3,9)$, $\adfPENT(3,3)$ and a 3-GDD of type $22^3 10^1$;\\
$\adfPENT(3, 46)$: 3 of $\adfPENT(3,9)$, $\adfPENT(3,13)$ and a 3-GDD of type $22^3 30^1$.
\end{proof}

We present an implementation of the above construction which gives pentagonal geometries PENT($3,9s-2$) and PENT($3,9s+1$), $s \geq 1$, directly from Steiner triple systems using a modification of a method which goes back to Bose \cite{B}, see also \cite{CR}. Denote a Steiner triple system by $(V,\mathcal{B})$ where $V$ is the set of points of cardinality $v$ and $\mathcal{B}$ is the family of triples of $V$ containing every pair of points precisely once.  It is well known that such systems exist if and only if $v\equiv$~1~or~3~(mod 6). The Bose construction is as follows. Let $(Q,\cdot)$ be a commutative idempotent quasigroup. These exist if and only if the cardinality of $Q$ is odd. Let $V=Q \times Z_3$.\\
Define two families of triples,\\
(i) $\mathcal{P}=\{\{(x,0),(x,1),(x,2)\}, x \in Q\}$ and\\
(ii) $\mathcal{S}= \{\{(x,i),(y,i),(x \cdot y, i+1)\}, x,y \in Q, x \neq y, i \in Z_3\}$.\\
Then $(V,\mathcal{P} \cup \mathcal{S})$ is a Steiner triple system of order 3 (mod 6).

We now use a Steiner triple system in order to construct the quasigroup. Let $(V,\mathcal{B})$ be a Steiner triple system where $v=6s+1$ or $6s+3$, $s \geq 1$. Put $Q=V$ and define an operation on $Q$ as follows, (i) $x\cdot x = x,~x \in V$ and (ii) $x \cdot y = z,~x,y \in V, x \neq y$ where $\{x,y,z\} \in \mathcal{B}$. The quasigroup so formed is called a \textit{Steiner quasigroup}.

Finally choose any point $a \in V$ and remove the triple $\{(a,0),(a,1),(a,2)\}$ from $\mathcal{P}$ and all further triples containing any of the three points $(a,0)$, $(a,1)$ or $(a,2)$ from $\mathcal{S}$. In this reduced structure the replication number of every point is $(3v-1)/2-3=9s-2$ if $v\equiv$~1~(mod 6) and $9s+1$ if $v\equiv$~3~(mod 6). We need to check that this structure is indeed a pentagonal geometry. Choose any point $(x,i)$, $x \neq a$. The triples containing $(x,i)$ which have been removed from the Steiner triple system are $\{(x,i),(a,i-1),(y,i-1)\}$, $\{(x,i),(a,i),(y,i+1)\}$ and $\{(x,i),(a,i+1),(y,i)\}$ where $\{a,x,y\} \in \mathcal{B}$ whilst the triple $\{(y,i-1),(y,i),(y,i+1)\}$ remains. The opposite line pairs are $\{(z,i-1),(z,i),(z,i+1)\}$, $\{(w,i-1),(w,i),(w,i+1)\}$ for all $z,w$ such that $\{a,z,w\} \in \mathcal{B}$.

The method can be extended to the case where $v\equiv$~5~(mod 6) by using a suitable \textit{pairwise balanced design}, PBD($v,K$). This is an ordered pair $(V,\mathcal{B})$ where $V$ is a set of points of cardinality $v$, $K$ is a set of positive integers and $\mathcal{B}$ is a family of lines or blocks such that if $B \in \mathcal{B}$ then $\vert B \vert \in K$. There exists a PBD($v, \{3,5^*\}$) for all $v\equiv$~5~(mod 6) where the asterisk on the 5 indicates that there is just one block of this cardinality, the \textit{distinguished block} \cite {F}. Now let $(V \cup Z_5, \mathcal{B})$ be a PBD($v,\{3,5^*\}$) where $v=6s+5$, $s \geq 1$ and $Z_5$ is the distinguished block. Put $Q=V \cup Z_5$ and define an operation on $Q$ as follows, (i) $x\cdot x = x,~x \in V$, (ii) $x \cdot y = z,~x,y \in Q, x \neq y$ where $\{x,y,z\} \in \mathcal{B}$ and (iii) $x \cdot y = (x+y)/2,~x,y, \in Z_5$. Now proceed as before. Choose any point $a \in V$ and remove the triple $\{(a,0),(a,1),(a,2)\}$ from $\mathcal{P}$ and all further triples containing any of the three points $(a,0)$, $(a,1)$ or $(a,2)$ from $\mathcal{S}$. We have a pentagonal geometry PENT($3,9s+4$), $s \geq 1$.


\section{Small systems}\label{sec:Small systems}
In this section we collect together the results of some of our computer calculations on the existence of pentagonal geometries PENT($3,r$) for $1 \leq r \leq 12$. We first note that both PENT($3,1$) and PENT($3,3$) are unique, being an opposite line pair and the Desargues configuration respectively.  There is no pentagonal geometry PENT($3,4$) by Theorem \ref{kplus1} and it was shown in \cite{BBDS} that PENT($3,6$) does not exist. From Theorem \ref{opplinepair} any pentagonal geometry PENT($3,7$) has either 0 or 3 opposite line pairs. The possibility of 0 opposite line pairs was eliminated in \cite{BBDS} by a computer search and the possibility of 3 opposite line pairs was considered in \cite{GS} where it was shown that the 36 lines not forming the opposite line pairs come from a Latin square of side 6.

In \cite{GS} it was proved that there is no pentagonal geometry PENT($3,9$) with two opposite line pairs. We now extend that result.

\begin{theorem}\label{2OLPs}
There is no pentagonal geometry $\adfPENT(3,r)$ with two opposite line pairs for $r \in \{7,9,10,12\}$.
\end{theorem}
\begin{proof}
A pentagonal geometry PENT($3,r$) has $2r+4$ points and $2r(r+2)/3$ lines. Suppose that there are two opposite line pairs. Call the points of one of the opposite line pairs type A and the points of the other opposite line pair type B. The remaining $2(r-4)$ points are type C. There are two lines of type AAA and two lines of type BBB (the opposite line pairs). The remaining  lines are of type ABC, ACC, BCC or CCC. Of the remaining pairs to be covered, 36 are type AB, $12(r-4)$ are type AC, $12(r-4)$ are type BC and $(r-4)(2r-9) - 3(r-4) = 2(r-4)(r-6)$ are type CC. So there are 36 lines of type ABC, and $(12(r-4)-36)/2 = 6(r-7)$ lines of both types ACC and BCC.
Hence there are $(2(r-4)(r-6)-12(r-7))/3 = (2r^2-32r+132)/3$ lines of type CCC. Now consider a point of type C. Its opposite line is of type CCC and since there are no opposite line pairs other than those of types AAA or BBB there are at least as many lines of type CCC as points of type C. Therefore if $(2r^2-32r+132)/3 < 2r-8$ giving $(2r-19)^2 < 49$, i.e. $r\in \{7,9,10,12\}$ no such pentagonal geometry exists.
\end{proof}

Returning now to the case where $r=9$, the only possibilities for the number of opposite line pairs is 0 or 1 and an example of each is given in Section \ref{sec:Block size 3}. In the latter case the deficiency graph must be the disjoint union of the complete bipartite graph $K_{3,3}$ and a cubic graph of girth at least 5 with 16 vertices. There are 49 such graphs \cite{M} and we have been able to construct a PENT($3,9$) for all of these possibilities.

When $r=10$, the number of opposite line pairs $n$ in a pentagonal geometry PENT($3,10$) satisfies the inequality $0 \leq n \leq 4$. We have shown in Section \ref{sec:Block size 3} that such pentagonal geometries for the two extreme values exist. Furthermore the value $n=3$ is impossible because a PENT($3,10$) having three opposite line pairs must also contain a fourth. The value $n=2$ is eliminated by the above theorem leaving only $n=1$. An example of a pentagonal geometry PENT($3,10$) with one opposite line pair is given in Appendix B. In such a system the deficiency graph must be the disjoint union of the graph $K_{3,3}$ and a cubic graph of girth at least 5 with 18 vertices. There are 455 such graphs \cite{M} and we have been able to construct a PENT($3,10$) for all but two of these graphs.

Finally, when $r=12$ the number of opposite line pairs $n$ in a PENT($3,12$) satisfies the inequality $0 \leq n \leq 3$. As with the case of $r=10$, we have such geometries with $n=0$ or 3 and $n=2$ is eliminated. A PENT($3,12$) with one opposite line pair is given in Appendix B.

The results in the previous section prove the existence of pentagonal geometries PENT($3,r$) with the maximum number of opposite line pairs and with no opposite line pair.  It is relevant to ask the question of what can be said about the existence of systems with a number of opposite line pairs between these two extremes. Now with the existence of a PENT($3,10$) with one opposite line pair as referred to above, it is very easy to prove that for any given number of opposite line pairs, there exists a PENT($3,r$) with precisely that number of opposite line pairs for all $r\equiv 0$~or~$1\pmod{3}$ and $r$ large enough. A bound is given in the statement of the next theorem.

\begin{theorem}\label{inter} 
For any non-negative integer $q$, there exists a pentagonal geometry $\adfPENT(3,r)$ having precisely $q$ opposite line pairs for all
$r\equiv 0$~or~$1\pmod{3}$ and $r \geq 12u+9$ where $u=\mathrm{max}(3,q)$.
\end{theorem}
\begin{proof}
First, observe that in Theorem \ref{gddgeneral} if the number of opposite line pairs in each of the pentagonal geometries PENT($3,r_i$) is $q_i$, $i=1,2,\ldots,n$, then the number of opposite line pairs in the pentagonal geometry PENT($3,\Sigma_{i=1}^n r_i + 2(n-1)$) is $\Sigma_{i=1}^n q_i$.\\
There exists a 3-GDD of type $24^u(2m)^1$ for all $u \geq 3$ and $m \leq 12(u-1)$. Let $u~\geq~$max($3,q$) and $m \in \{11,12,14,15,17,18,20,21\}$. On $q$ of the groups of cardinality 24 of the GDD construct a pentagonal geometry PENT($3,10$) with one opposite line pair and on the remaining $u-q$ groups construct a PENT($3,10$) with no opposite line pair. On the group of cardinality $2m$ construct a pentagonal geometry PENT($3,m-2$) with no opposite line pair. Adjoin these geometries to the blocks of the GDD. We have a pentagonal geometry PENT($3,12u+m-2$).
\end{proof}     


\section{Block size 4}\label{sec:Block size 4}
A pentagonal geometry with block size 4, $\adfPENT(4,r)$, has $3r+5$ points and $r(3r+5)/4$ lines, and therefore a necessary existence condition is that $r \equiv 0 \textrm{~or~} 1 \adfmod{4}$.
Using the basic construction of Theorem~\ref{g^u} with the degenerate $\adfPENT(4,1)$ and 4-GDDs of type $8^{3t+1}$ it is proved in \cite{GS} that
there exists a pentagonal geometry $\adfPENT(4,r)$ for all $r \equiv 1 \adfmod{8}$. In this section we extend those results considerably. We are able to solve the existence spectrum problem for $\adfPENT(4,r)$ completely when $r \equiv 1 \adfmod{4}$ and with a small number of possible exceptions when $r \equiv 0 \adfmod{4}$. These are Theorems~\ref{pent4,1} and \ref{pent4,0} respectively. For these theorems we will need pentagonal geometries PENT($4,r$) for $r \in \{13, 20, 24\}$ and these are given below.~\\ 

{\noindent\boldmath $\adfPENT(4, 13)$}.\\
With point set $Z_{44}$, the 143 lines are generated from\\
{\small\begin{tabular}{llll}
\{20,24,25,31\}, & \{17,20,29,38\}, & \{9,14,23,34\}, & \{7,16,26,43\},\\
\{0,12,34,38\}, & \{0,3,8,23\}, & \{0,13,42,43\}, & \{0,16,33,37\},\\
\{0,2,29,30\}, & \{0,6,14,19\}, & \{1,9,35,42\}, & \{1,11,14,43\},\\
\{1,3,19,25\}. & ~ & ~ & ~
\end{tabular}}\\
under the action of the mapping $x \mapsto x + 4 \adfmod{44}$.\\
The deficiency graph is connected and has girth 5.\\

{\noindent\boldmath $\adfPENT(4, 20)$}.\\
With point set $Z_{65}$, the 325 lines are generated from\\
{\small\begin{tabular}{llll}
\{22,43,51,59\}, & \{15,26,41,47\}, & \{21,32,37,45\}, & \{8,9,25,63\},\\
\{10,34,39,63\}, & \{15,19,30,57\}, & \{25,48,23,6\}, & \{57,18,32,64\},\\
\{22,7,13,4\}, & \{8,58,12,57\}, & \{13,46,58,50\}, & \{1,8,38,62\},\\
\{1,23,49,57\}, & \{0,31,33,62\}, & \{0,13,20,64\}, & \{0,3,16,21\},\\
\{0,14,18,39\}, & \{1,28,31,59\}, & \{1,4,14,27\}, & \{1,11,34,54\},\\
\{0,17,19,34\}, & \{0,9,36,56\}, & \{0,1,2,60\}, & \{0,10,35,47\},\\
\{1,42,52,64\}. & ~ & ~ & ~
\end{tabular}}\\
under the action of the mapping $x \mapsto x + 5 \adfmod{65}$.\\
The deficiency graph is connected and has girth 5.\\

\newpage
{\noindent\boldmath $\adfPENT(4, 24)$}.\\
With point set $Z_{77}$, the 462 lines are generated from\\
{\small\begin{tabular}{llll}
\{15,28,37,49\}, & \{4,58,59,63\}, & \{22,42,54,75\}, & \{6,22,38,45\},\\
\{1,39,46,61\}, & \{9,25,30,69\}, & \{3,19,20,69\}, & \{48,13,11,58\},\\
\{39,73,5,8\}, & \{18,13,32,59\}, & \{75,34,56,70\}, & \{27,74,61,21\},\\
\{18,74,35,70\}, & \{54,36,66,47\}, & \{9,40,46,75\}, & \{18,49,46,71\},\\
\{49,19,50,59\}, & \{27,2,47,56\}, & \{68,5,20,38\}, & \{6,7,59,14\},\\
\{1,18,22,69\}, & \{0,20,26,75\}, & \{2,11,12,17\}, & \{1,9,60,68\},\\
\{0,38,59,61\}, & \{2,40,66,67\}, & \{0,3,13,67\}, & \{2,5,46,65\},\\
\{1,2,26,64\}, & \{0,50,54,62\}, & \{0,24,30,32\}, & \{0,11,27,31\},\\
\{1,38,52,67\}, & \{1,11,20,76\}, & \{0,2,8,36\}, & \{1,3,51,55\},\\
\{1,8,34,37\}, & \{1,45,65,73\}, & \{0,29,34,71\}, & \{0,16,44,66\},\\
\{2,13,20,44\}, & \{0,17,51,58\}. & ~ & ~
\end{tabular}}\\
under the action of the mapping $x \mapsto x + 7 \adfmod{77}$.\\
The deficiency graph is connected and has girth 6.\\

For the existence of the 4-GDDs of type $8^u$ and of type $8^u m^1$ used in the following proofs, see \cite{BSH} and \cite{S} respectively or \cite[Theorem 7.1]{WG}.

\begin{theorem}\label{pent4,1}
There exist pentagonal geometries $\adfPENT(4,r)$ for all $r \equiv 1 \adfmod{4}$, except for $r = 5$.
\end{theorem}
\begin{proof}
For $r \equiv 1 \adfmod{8}$, in Theorem \ref{g^u} let $k=4$ and $r=1$. There exists a pentagonal geometry PENT($4,1$) and a 4-GDD of type
$8^{3t+1}$, $t \geq 1$. Hence there exists a PENT($4,8t+1$), $t \geq 1$. 

For $r \equiv 5 \adfmod{8}$, in Theorem \ref{g^um^1} let $k=4$, $r=1$ and $s=13$. There exist pentagonal geometries PENT($4,1$) and PENT($4,13$) and a 4-GDD of type $8^{3t}44^1$, $t \geq 4$. Hence there exists a PENT($4,8t+13$), $t \geq 4$.

For PENT($4,r$), $r \in \{21,29,37\}$ see Appendix C; PENT($4,5$) does not exist by Theorem \ref{kplus1}.
\end{proof} 

\begin{theorem}
\label{pent4,0}
There exist pentagonal geometries $\adfPENT(4,r)$ for all $r \equiv 0 \adfmod{4}$, except for $r = 4$ and except possibly for
\begin{center}
$r \in \{8,12,16,28,32,36,44,48,56,64,72\}$.
\end{center}
\end{theorem}
\begin{proof}
For $r \equiv 4 \adfmod{8}$, in Theorem \ref{g^um^1} let $k=4$, $r=1$ and $s=20$. There exist pentagonal geometries PENT($4,1$) and PENT($4,20$) and a 4-GDD of type $8^{3t}65^1$, $t \geq 6$. Hence there exists a PENT($4,8t+20$), $t \geq 6$. 

For $r \equiv 0 \adfmod{8}$, in Theorem \ref{g^um^1} proceed as in the case above but let $s=24$. There exists a 4-GDD of type $8^{3t}77^1$, $t \geq 7$ and hence there exists a PENT($4,8t+24$), $t \geq 7$.

For PENT($4,r$), $r \in \{40,52,60\}$ see Appendix C; PENT($4,4$) does not exist by Theorem \ref{kk}. This accounts for all relevant values of $r$ except those listed as possible exceptions.
\end{proof}


\section*{Appendix A}\label{app:PENT3}
{\noindent\boldmath $\adfPENT(3, 10)$}.\\
With point set $Z_{24}$, the 80 lines are
 
{\small
 $\{0, 3, 4\}$, $\{0, 5, 17\}$, $\{0, 6, 14\}$, $\{0, 7, 18\}$, $\{0, 8, 13\}$, $\{0, 9, 16\}$,\adfsplit
 $\{0, 10, 19\}$, $\{0, 11, 15\}$, $\{0, 12, 21\}$, $\{0, 20, 23\}$, $\{1, 2, 22\}$, $\{1, 3, 18\}$,\adfsplit
 $\{1, 4, 23\}$, $\{1, 5, 14\}$, $\{1, 6, 19\}$, $\{1, 7, 9\}$, $\{1, 8, 17\}$, $\{1, 10, 21\}$,\adfsplit
 $\{1, 12, 20\}$, $\{1, 13, 16\}$, $\{2, 5, 6\}$, $\{2, 7, 12\}$, $\{2, 8, 18\}$, $\{2, 9, 15\}$,\adfsplit
 $\{2, 10, 20\}$, $\{2, 11, 19\}$, $\{2, 13, 17\}$, $\{2, 14, 23\}$, $\{2, 16, 21\}$, $\{3, 5, 11\}$,\adfsplit
 $\{3, 6, 20\}$, $\{3, 7, 16\}$, $\{3, 8, 14\}$, $\{3, 9, 21\}$, $\{3, 10, 15\}$, $\{3, 12, 23\}$,\adfsplit
 $\{3, 19, 22\}$, $\{4, 7, 8\}$, $\{4, 9, 12\}$, $\{4, 10, 18\}$, $\{4, 11, 16\}$, $\{4, 13, 20\}$,\adfsplit
 $\{4, 14, 21\}$, $\{4, 15, 19\}$, $\{4, 17, 22\}$, $\{5, 7, 13\}$, $\{5, 8, 20\}$, $\{5, 9, 18\}$,\adfsplit
 $\{5, 10, 16\}$, $\{5, 12, 22\}$, $\{5, 21, 23\}$, $\{6, 9, 10\}$, $\{6, 11, 13\}$, $\{6, 12, 18\}$,\adfsplit
 $\{6, 15, 22\}$, $\{6, 16, 23\}$, $\{6, 17, 21\}$, $\{7, 10, 22\}$, $\{7, 11, 20\}$, $\{7, 14, 19\}$,\adfsplit
 $\{7, 15, 23\}$, $\{8, 11, 12\}$, $\{8, 15, 21\}$, $\{8, 16, 22\}$, $\{8, 19, 23\}$, $\{9, 11, 17\}$,\adfsplit
 $\{9, 13, 22\}$, $\{9, 14, 20\}$, $\{10, 13, 14\}$, $\{10, 17, 23\}$, $\{11, 14, 22\}$, $\{11, 18, 23\}$,\adfsplit
 $\{12, 15, 16\}$, $\{12, 17, 19\}$, $\{13, 15, 18\}$, $\{13, 19, 21\}$, $\{14, 17, 18\}$, $\{15, 17, 20\}$,\adfsplit
 $\{16, 19, 20\}$, $\{18, 21, 22\}$.
}

\noindent The deficiency graph is connected and has girth 6.~\\

\newpage
{\noindent\boldmath $\adfPENT(3, 16)$}.\\
With point set $Z_{36}$, the 192 lines are generated from

{\small
 $\{1, 15, 21\}$, $\{0, 14, 22\}$, $\{6, 25, 31\}$, $\{5, 24, 30\}$, $\{18, 19, 22\}$, $\{3, 8, 11\}$,\adfsplit
 $\{2, 15, 28\}$, $\{14, 21, 28\}$, $\{5, 23, 34\}$, $\{23, 24, 31\}$, $\{20, 25, 28\}$, $\{5, 32, 33\}$,\adfsplit
 $\{19, 25, 16\}$, $\{3, 10, 18\}$, $\{12, 4, 31\}$, $\{10, 12, 32\}$, $\{34, 33, 29\}$, $\{10, 16, 22\}$,\adfsplit
 $\{17, 5, 14\}$, $\{22, 5, 7\}$, $\{13, 25, 14\}$, $\{29, 3, 6\}$, $\{7, 25, 17\}$, $\{30, 20, 9\}$,\adfsplit
 $\{0, 10, 13\}$, $\{0, 5, 25\}$, $\{1, 5, 6\}$, $\{1, 9, 11\}$, $\{3, 17, 31\}$, $\{0, 8, 29\}$,\adfsplit
 $\{2, 17, 33\}$, $\{4, 5, 28\}$, $\{4, 23, 29\}$, $\{1, 16, 20\}$, $\{1, 3, 27\}$, $\{11, 13, 20\}$,\adfsplit
 $\{1, 10, 23\}$, $\{1, 30, 33\}$, $\{2, 21, 34\}$, $\{3, 7, 34\}$, $\{6, 20, 22\}$, $\{3, 22, 33\}$,\adfsplit
 $\{2, 22, 23\}$, $\{0, 16, 32\}$, $\{7, 20, 33\}$, $\{3, 23, 28\}$, $\{4, 6, 11\}$, $\{0, 31, 33\}$,\adfsplit
 $\{6, 23, 33\}$, $\{4, 9, 21\}$, $\{0, 9, 27\}$, $\{11, 14, 15\}$, $\{0, 3, 4\}$, $\{2, 4, 30\}$,\adfsplit
 $\{0, 18, 30\}$, $\{0, 2, 12\}$, $\{0, 11, 23\}$, $\{6, 8, 19\}$, $\{2, 11, 18\}$, $\{7, 11, 31\}$,\adfsplit
 $\{2, 20, 27\}$, $\{3, 19, 20\}$, $\{2, 8, 32\}$, $\{2, 7, 26\}$
}

under the action of the mapping $x \mapsto x + 12 \adfmod{36}$.\\
The deficiency graph is connected and has girth 5.~\\

{\noindent\boldmath $\adfPENT(3, 19)$}.\\
With point set $Z_{42}$, the 266 lines are generated from

{\small
 $\{18, 24, 34\}$, $\{13, 31, 35\}$, $\{11, 16, 23\}$, $\{15, 22, 33\}$, $\{12, 27, 32\}$, $\{13, 26, 38\}$,\adfsplit
 $\{30, 33, 34\}$, $\{15, 14, 11\}$, $\{40, 37, 16\}$, $\{3, 20, 28\}$, $\{1, 10, 15\}$, $\{38, 28, 34\}$,\adfsplit
 $\{16, 19, 5\}$, $\{26, 24, 29\}$, $\{37, 21, 35\}$, $\{13, 9, 19\}$, $\{20, 27, 6\}$, $\{20, 22, 38\}$,\adfsplit
 $\{0, 8, 28\}$, $\{1, 16, 28\}$, $\{0, 22, 23\}$, $\{1, 17, 34\}$, $\{0, 11, 40\}$, $\{3, 16, 35\}$,\adfsplit
 $\{0, 17, 32\}$, $\{0, 1, 30\}$, $\{0, 26, 37\}$, $\{0, 7, 9\}$, $\{0, 25, 35\}$, $\{1, 9, 20\}$,\adfsplit
 $\{0, 31, 38\}$, $\{1, 2, 38\}$, $\{0, 27, 29\}$, $\{0, 19, 39\}$, $\{3, 11, 12\}$, $\{3, 23, 29\}$,\adfsplit
 $\{2, 17, 35\}$, $\{2, 15, 21\}$
}

under the action of the mapping $x \mapsto x + 6 \adfmod{42}$.\\
The deficiency graph is connected and has girth 5.~\\

{\noindent\boldmath $\adfPENT(3, 22)$}.\\
With point set $Z_{48}$, the 352 lines are

{\small
 $\{0, 3, 4\}$, $\{0, 5, 45\}$, $\{0, 6, 15\}$, $\{0, 7, 19\}$, $\{0, 8, 27\}$, $\{0, 9, 18\}$,\adfsplit
 $\{0, 10, 32\}$, $\{0, 11, 22\}$, $\{0, 12, 33\}$, $\{0, 13, 38\}$, $\{0, 14, 21\}$, $\{0, 16, 34\}$,\adfsplit
 $\{0, 17, 37\}$, $\{0, 20, 31\}$, $\{0, 23, 42\}$, $\{0, 24, 40\}$, $\{0, 25, 36\}$, $\{0, 26, 39\}$,\adfsplit
 $\{0, 28, 43\}$, $\{0, 29, 35\}$, $\{0, 30, 41\}$, $\{0, 44, 47\}$, $\{1, 2, 46\}$, $\{1, 3, 27\}$,\adfsplit
 $\{1, 4, 9\}$, $\{1, 6, 21\}$, $\{1, 7, 32\}$, $\{1, 8, 18\}$, $\{1, 10, 42\}$, $\{1, 11, 38\}$,\adfsplit
 $\{1, 12, 34\}$, $\{1, 13, 24\}$, $\{1, 14, 35\}$, $\{1, 15, 36\}$, $\{1, 16, 25\}$, $\{1, 17, 28\}$,\adfsplit
 $\{1, 19, 26\}$, $\{1, 20, 30\}$, $\{1, 22, 43\}$, $\{1, 23, 37\}$, $\{1, 29, 39\}$, $\{1, 31, 47\}$,\adfsplit
 $\{1, 33, 40\}$, $\{1, 41, 44\}$, $\{2, 5, 6\}$, $\{2, 7, 47\}$, $\{2, 8, 42\}$, $\{2, 9, 45\}$,\adfsplit
 $\{2, 10, 39\}$, $\{2, 11, 23\}$, $\{2, 12, 37\}$, $\{2, 13, 27\}$, $\{2, 14, 43\}$, $\{2, 15, 24\}$,\adfsplit
 $\{2, 16, 30\}$, $\{2, 17, 35\}$, $\{2, 18, 36\}$, $\{2, 19, 41\}$, $\{2, 20, 29\}$, $\{2, 21, 34\}$,\adfsplit
 $\{2, 22, 31\}$, $\{2, 25, 32\}$, $\{2, 26, 33\}$, $\{2, 28, 40\}$, $\{2, 38, 44\}$, $\{3, 5, 35\}$,\adfsplit
 $\{3, 6, 11\}$, $\{3, 8, 30\}$, $\{3, 9, 23\}$, $\{3, 10, 20\}$, $\{3, 12, 19\}$, $\{3, 13, 25\}$,\adfsplit
 $\{3, 14, 24\}$, $\{3, 15, 29\}$, $\{3, 16, 45\}$, $\{3, 17, 36\}$, $\{3, 18, 32\}$, $\{3, 21, 42\}$,\adfsplit
 $\{3, 22, 33\}$, $\{3, 26, 38\}$, $\{3, 28, 37\}$, $\{3, 31, 44\}$, $\{3, 34, 40\}$, $\{3, 39, 41\}$,\adfsplit
 $\{3, 43, 46\}$, $\{4, 7, 8\}$, $\{4, 10, 25\}$, $\{4, 11, 47\}$, $\{4, 12, 39\}$, $\{4, 13, 30\}$,\adfsplit
 $\{4, 14, 46\}$, $\{4, 15, 38\}$, $\{4, 16, 40\}$, $\{4, 17, 26\}$, $\{4, 18, 34\}$, $\{4, 19, 33\}$,\adfsplit
 $\{4, 20, 45\}$, $\{4, 21, 28\}$, $\{4, 22, 29\}$, $\{4, 23, 43\}$, $\{4, 24, 36\}$, $\{4, 27, 44\}$,\adfsplit
 $\{4, 31, 37\}$, $\{4, 32, 42\}$, $\{4, 35, 41\}$, $\{5, 7, 38\}$, $\{5, 8, 13\}$, $\{5, 10, 24\}$,\adfsplit
 $\{5, 11, 17\}$, $\{5, 12, 42\}$, $\{5, 14, 32\}$, $\{5, 15, 47\}$, $\{5, 16, 37\}$, $\{5, 18, 28\}$,\adfsplit
 $\{5, 19, 29\}$, $\{5, 20, 39\}$, $\{5, 21, 41\}$, $\{5, 22, 44\}$, $\{5, 23, 34\}$, $\{5, 25, 27\}$,\adfsplit
 $\{5, 26, 40\}$, $\{5, 30, 43\}$, $\{5, 31, 33\}$, $\{5, 36, 46\}$, $\{6, 9, 10\}$, $\{6, 12, 24\}$,\adfsplit
 $\{6, 13, 28\}$, $\{6, 14, 27\}$, $\{6, 16, 32\}$, $\{6, 17, 41\}$, $\{6, 18, 38\}$, $\{6, 19, 34\}$,\adfsplit
 $\{6, 20, 26\}$, $\{6, 22, 40\}$, $\{6, 23, 47\}$, $\{6, 25, 45\}$, $\{6, 29, 36\}$, $\{6, 30, 44\}$,\adfsplit
 $\{6, 31, 43\}$, $\{6, 33, 42\}$, $\{6, 35, 37\}$, $\{6, 39, 46\}$, $\{7, 9, 26\}$, $\{7, 10, 15\}$,\adfsplit
 $\{7, 12, 28\}$, $\{7, 13, 45\}$, $\{7, 14, 39\}$, $\{7, 16, 22\}$, $\{7, 17, 42\}$, $\{7, 18, 43\}$,\adfsplit
 $\{7, 20, 35\}$, $\{7, 21, 40\}$, $\{7, 23, 30\}$, $\{7, 24, 37\}$, $\{7, 25, 44\}$, $\{7, 27, 36\}$,\adfsplit
 $\{7, 29, 34\}$, $\{7, 31, 41\}$, $\{7, 33, 46\}$, $\{8, 11, 12\}$, $\{8, 14, 40\}$, $\{8, 15, 37\}$,\adfsplit
 $\{8, 16, 33\}$, $\{8, 17, 44\}$, $\{8, 19, 31\}$, $\{8, 20, 38\}$, $\{8, 21, 32\}$, $\{8, 22, 47\}$,\adfsplit
 $\{8, 23, 36\}$, $\{8, 24, 41\}$, $\{8, 25, 34\}$, $\{8, 26, 35\}$, $\{8, 28, 46\}$, $\{8, 29, 43\}$,\adfsplit
 $\{8, 39, 45\}$, $\{9, 11, 24\}$, $\{9, 12, 17\}$, $\{9, 14, 30\}$, $\{9, 15, 32\}$, $\{9, 16, 46\}$,\adfsplit
 $\{9, 19, 39\}$, $\{9, 20, 28\}$, $\{9, 21, 31\}$, $\{9, 22, 38\}$, $\{9, 25, 40\}$, $\{9, 27, 43\}$,\adfsplit
 $\{9, 29, 42\}$, $\{9, 33, 35\}$, $\{9, 34, 41\}$, $\{9, 36, 47\}$, $\{9, 37, 44\}$, $\{10, 13, 14\}$,\adfsplit
 $\{10, 16, 35\}$, $\{10, 17, 22\}$, $\{10, 18, 40\}$, $\{10, 19, 43\}$, $\{10, 21, 47\}$, $\{10, 23, 28\}$,\adfsplit
 $\{10, 26, 36\}$, $\{10, 27, 41\}$, $\{10, 29, 44\}$, $\{10, 30, 37\}$, $\{10, 31, 46\}$, $\{10, 33, 38\}$,\adfsplit
 $\{10, 34, 45\}$, $\{11, 13, 46\}$, $\{11, 14, 19\}$, $\{11, 16, 44\}$, $\{11, 18, 30\}$, $\{11, 20, 33\}$,\adfsplit
 $\{11, 21, 37\}$, $\{11, 25, 35\}$, $\{11, 26, 34\}$, $\{11, 27, 42\}$, $\{11, 28, 41\}$, $\{11, 29, 45\}$,\adfsplit
 $\{11, 31, 40\}$, $\{11, 32, 39\}$, $\{11, 36, 43\}$, $\{12, 15, 16\}$, $\{12, 18, 26\}$, $\{12, 20, 44\}$,\adfsplit
 $\{12, 21, 27\}$, $\{12, 22, 45\}$, $\{12, 23, 38\}$, $\{12, 25, 43\}$, $\{12, 29, 41\}$, $\{12, 30, 46\}$,\adfsplit
 $\{12, 31, 36\}$, $\{12, 32, 40\}$, $\{12, 35, 47\}$, $\{13, 15, 26\}$, $\{13, 16, 21\}$, $\{13, 18, 35\}$,\adfsplit
 $\{13, 19, 40\}$, $\{13, 20, 32\}$, $\{13, 22, 37\}$, $\{13, 23, 29\}$, $\{13, 31, 42\}$, $\{13, 33, 39\}$,\adfsplit
 $\{13, 34, 43\}$, $\{13, 36, 44\}$, $\{13, 41, 47\}$, $\{14, 17, 18\}$, $\{14, 20, 36\}$, $\{14, 22, 34\}$,\adfsplit
 $\{14, 23, 41\}$, $\{14, 25, 31\}$, $\{14, 26, 44\}$, $\{14, 28, 38\}$, $\{14, 29, 47\}$, $\{14, 33, 45\}$,\adfsplit
 $\{14, 37, 42\}$, $\{15, 17, 31\}$, $\{15, 18, 23\}$, $\{15, 20, 42\}$, $\{15, 21, 43\}$, $\{15, 22, 41\}$,\adfsplit
 $\{15, 25, 39\}$, $\{15, 27, 40\}$, $\{15, 28, 35\}$, $\{15, 30, 45\}$, $\{15, 33, 44\}$, $\{15, 34, 46\}$,\adfsplit
 $\{16, 19, 20\}$, $\{16, 23, 39\}$, $\{16, 24, 47\}$, $\{16, 26, 42\}$, $\{16, 27, 29\}$, $\{16, 28, 36\}$,\adfsplit
 $\{16, 31, 38\}$, $\{16, 41, 43\}$, $\{17, 19, 46\}$, $\{17, 20, 25\}$, $\{17, 23, 45\}$, $\{17, 24, 39\}$,\adfsplit
 $\{17, 27, 33\}$, $\{17, 29, 38\}$, $\{17, 30, 40\}$, $\{17, 32, 43\}$, $\{17, 34, 47\}$, $\{18, 21, 22\}$,\adfsplit
 $\{18, 24, 42\}$, $\{18, 25, 41\}$, $\{18, 27, 45\}$, $\{18, 29, 31\}$, $\{18, 33, 47\}$, $\{18, 37, 46\}$,\adfsplit
 $\{18, 39, 44\}$, $\{19, 21, 38\}$, $\{19, 22, 27\}$, $\{19, 24, 32\}$, $\{19, 25, 37\}$, $\{19, 28, 44\}$,\adfsplit
 $\{19, 30, 47\}$, $\{19, 35, 45\}$, $\{19, 36, 42\}$, $\{20, 23, 24\}$, $\{20, 27, 34\}$, $\{20, 37, 43\}$,\adfsplit
 $\{20, 40, 47\}$, $\{20, 41, 46\}$, $\{21, 23, 33\}$, $\{21, 24, 29\}$, $\{21, 26, 46\}$, $\{21, 30, 39\}$,\adfsplit
 $\{21, 35, 44\}$, $\{21, 36, 45\}$, $\{22, 25, 26\}$, $\{22, 28, 39\}$, $\{22, 30, 36\}$, $\{22, 32, 46\}$,\adfsplit
 $\{22, 35, 42\}$, $\{23, 25, 46\}$, $\{23, 26, 31\}$, $\{23, 32, 44\}$, $\{23, 35, 40\}$, $\{24, 27, 28\}$,\adfsplit
 $\{24, 30, 38\}$, $\{24, 31, 45\}$, $\{24, 33, 43\}$, $\{24, 34, 44\}$, $\{24, 35, 46\}$, $\{25, 28, 33\}$,\adfsplit
 $\{25, 30, 42\}$, $\{25, 38, 47\}$, $\{26, 29, 30\}$, $\{26, 32, 41\}$, $\{26, 37, 47\}$, $\{26, 43, 45\}$,\adfsplit
 $\{27, 30, 35\}$, $\{27, 32, 47\}$, $\{27, 37, 39\}$, $\{27, 38, 46\}$, $\{28, 31, 32\}$, $\{28, 34, 42\}$,\adfsplit
 $\{28, 45, 47\}$, $\{29, 32, 37\}$, $\{29, 40, 46\}$, $\{30, 33, 34\}$, $\{31, 34, 39\}$, $\{32, 35, 36\}$,\adfsplit
 $\{32, 38, 45\}$, $\{33, 36, 41\}$, $\{34, 37, 38\}$, $\{35, 38, 43\}$, $\{36, 39, 40\}$, $\{37, 40, 45\}$,\adfsplit
 $\{38, 41, 42\}$, $\{39, 42, 47\}$, $\{40, 43, 44\}$, $\{42, 45, 46\}$.
}

\noindent The deficiency graph is connected and has girth 5.~\\

\newpage
{\noindent\boldmath $\adfPENT(3, 24)$}.\\
With point set $Z_{52}$, the 416 lines are generated from

{\small
 $\{17, 21, 43\}$, $\{32, 35, 36\}$, $\{18, 38, 43\}$, $\{12, 14, 21\}$, $\{51, 27, 20\}$, $\{25, 11, 49\}$,\adfsplit
 $\{4, 23, 36\}$, $\{4, 10, 39\}$, $\{35, 37, 3\}$, $\{31, 30, 36\}$, $\{36, 47, 7\}$, $\{6, 10, 33\}$,\adfsplit
 $\{1, 3, 11\}$, $\{1, 43, 47\}$, $\{0, 14, 27\}$, $\{2, 11, 47\}$, $\{0, 15, 34\}$, $\{1, 7, 22\}$,\adfsplit
 $\{1, 14, 31\}$, $\{2, 23, 26\}$, $\{0, 10, 22\}$, $\{0, 18, 26\}$, $\{0, 5, 38\}$, $\{0, 33, 42\}$,\adfsplit
 $\{0, 1, 50\}$, $\{0, 30, 41\}$, $\{1, 6, 41\}$, $\{1, 2, 17\}$, $\{0, 37, 45\}$, $\{0, 29, 49\}$,\adfsplit
 $\{0, 12, 25\}$, $\{0, 8, 24\}$
}

under the action of the mapping $x \mapsto x + 4 \adfmod{52}$.\\
The deficiency graph is connected and has girth 7.~\\

{\noindent\boldmath $\adfPENT(3, 25)$}.\\
With point set $Z_{54}$, the 450 lines are generated from

{\small
 $\{16, 21, 43\}$, $\{12, 21, 46\}$, $\{33, 35, 47\}$, $\{26, 36, 37\}$, $\{13, 35, 42\}$, $\{14, 26, 28\}$,\adfsplit
 $\{48, 26, 16\}$, $\{24, 18, 0\}$, $\{22, 34, 3\}$, $\{1, 51, 2\}$, $\{11, 49, 14\}$, $\{34, 48, 35\}$,\adfsplit
 $\{38, 11, 22\}$, $\{14, 6, 39\}$, $\{36, 38, 51\}$, $\{16, 8, 13\}$, $\{37, 16, 23\}$, $\{3, 41, 4\}$,\adfsplit
 $\{19, 9, 25\}$, $\{49, 7, 20\}$, $\{42, 41, 33\}$, $\{1, 27, 15\}$, $\{16, 35, 20\}$, $\{33, 40, 46\}$,\adfsplit
 $\{8, 42, 15\}$, $\{48, 31, 40\}$, $\{29, 25, 8\}$, $\{2, 50, 28\}$, $\{13, 0, 28\}$, $\{2, 5, 20\}$,\adfsplit
 $\{1, 9, 52\}$, $\{4, 29, 34\}$, $\{1, 22, 40\}$, $\{0, 10, 12\}$, $\{2, 22, 39\}$, $\{0, 4, 17\}$,\adfsplit
 $\{0, 38, 39\}$, $\{3, 9, 27\}$, $\{2, 11, 21\}$, $\{0, 3, 14\}$, $\{0, 26, 50\}$, $\{1, 8, 31\}$,\adfsplit
 $\{0, 11, 19\}$, $\{0, 5, 7\}$, $\{0, 31, 49\}$, $\{0, 29, 35\}$, $\{0, 23, 51\}$, $\{1, 3, 35\}$,\adfsplit
 $\{5, 9, 29\}$, $\{1, 11, 29\}$
}

under the action of the mapping $x \mapsto x + 6 \adfmod{54}$.\\
The deficiency graph is connected and has girth 5.~\\

{\noindent\boldmath $\adfPENT(3, 27)$}.\\
With point set $Z_{58}$, the 522 lines are generated from

{\small
 $\{6, 45, 52\}$, $\{3, 14, 57\}$, $\{21, 27, 37\}$, $\{13, 33, 6\}$, $\{7, 36, 56\}$, $\{16, 21, 18\}$,\adfsplit
 $\{6, 24, 34\}$, $\{13, 25, 46\}$, $\{0, 1, 31\}$, $\{0, 4, 26\}$, $\{0, 11, 14\}$, $\{0, 8, 21\}$,\adfsplit
 $\{0, 17, 53\}$, $\{0, 16, 35\}$, $\{1, 15, 33\}$, $\{0, 33, 41\}$, $\{0, 15, 24\}$, $\{1, 2, 25\}$
}

under the action of the mapping $x \mapsto x + 2 \adfmod{58}$.\\
The deficiency graph is connected and has girth 6.~\\

\section*{Appendix B}\label{app:SMALLSYS}
{\noindent\boldmath $\adfPENT(3, 10)$}.\\
With point set $Z_{24}$, the 80 lines are
 
{\small
 $\{0, 4, 5\}$, $\{0, 6, 7\}$, $\{0, 8, 9\}$, $\{0, 10, 14\}$, $\{0, 11, 22\}$, $\{0, 12, 20\}$,\adfsplit
 $\{0, 13, 23\}$, $\{0, 15, 19\}$, $\{0, 16, 21\}$, $\{0, 17, 18\}$, $\{1, 2, 3\}$, $\{1, 6, 21\}$,\adfsplit
 $\{1, 7, 15\}$, $\{1, 8, 20\}$, $\{1, 9, 22\}$, $\{1, 10, 11\}$, $\{1, 12, 13\}$, $\{1, 14, 23\}$,\adfsplit
 $\{1, 16, 18\}$, $\{1, 17, 19\}$, $\{2, 4, 19\}$, $\{2, 5, 18\}$, $\{2, 8, 21\}$, $\{2, 9, 23\}$,\adfsplit
 $\{2, 10, 12\}$, $\{2, 11, 14\}$, $\{2, 13, 17\}$, $\{2, 15, 20\}$, $\{2, 16, 22\}$, $\{3, 4, 23\}$,\adfsplit
 $\{3, 5, 20\}$, $\{3, 6, 16\}$, $\{3, 7, 21\}$, $\{3, 10, 15\}$, $\{3, 11, 18\}$, $\{3, 12, 19\}$,\adfsplit
 $\{3, 13, 14\}$, $\{3, 17, 22\}$, $\{4, 6, 8\}$, $\{4, 7, 16\}$, $\{4, 9, 12\}$, $\{4, 13, 20\}$,\adfsplit
 $\{4, 14, 22\}$, $\{4, 15, 18\}$, $\{4, 17, 21\}$, $\{5, 6, 17\}$, $\{5, 7, 19\}$, $\{5, 8, 11\}$,\adfsplit
 $\{5, 9, 16\}$, $\{5, 10, 22\}$, $\{5, 14, 21\}$, $\{5, 15, 23\}$, $\{6, 9, 20\}$, $\{6, 11, 23\}$,\adfsplit
 $\{6, 13, 19\}$, $\{6, 14, 18\}$, $\{6, 15, 22\}$, $\{7, 8, 23\}$, $\{7, 9, 17\}$, $\{7, 10, 20\}$,\adfsplit
 $\{7, 12, 22\}$, $\{7, 13, 18\}$, $\{8, 12, 18\}$, $\{8, 13, 22\}$, $\{8, 14, 19\}$, $\{8, 16, 17\}$,\adfsplit
 $\{9, 10, 18\}$, $\{9, 11, 19\}$, $\{9, 15, 21\}$, $\{10, 13, 21\}$, $\{10, 16, 19\}$, $\{10, 17, 23\}$,\adfsplit
 $\{11, 12, 21\}$, $\{11, 13, 15\}$, $\{11, 17, 20\}$, $\{12, 14, 15\}$, $\{12, 16, 23\}$, $\{14, 16, 20\}$,\adfsplit
 $\{18, 19, 20\}$, $\{21, 22, 23\}$.
}

\noindent The opposite line pair is $\{18,19,20\}$, $\{21,22,23\}$.~\\

{\noindent\boldmath $\adfPENT(3, 12)$}.\\
With point set $Z_{28}$, the 112 lines are
 
{\small
 $\{0, 4, 5\}$, $\{0, 6, 7\}$, $\{0, 8, 9\}$, $\{0, 10, 27\}$, $\{0, 11, 15\}$, $\{0, 12, 20\}$,\adfsplit
 $\{0, 13, 24\}$, $\{0, 14, 26\}$, $\{0, 16, 25\}$, $\{0, 17, 21\}$, $\{0, 18, 22\}$, $\{0, 19, 23\}$,\adfsplit
 $\{1, 2, 3\}$, $\{1, 6, 15\}$, $\{1, 7, 20\}$, $\{1, 8, 16\}$, $\{1, 9, 23\}$, $\{1, 10, 11\}$,\adfsplit
 $\{1, 12, 13\}$, $\{1, 14, 22\}$, $\{1, 17, 24\}$, $\{1, 18, 25\}$, $\{1, 19, 27\}$, $\{1, 21, 26\}$,\adfsplit
 $\{2, 4, 8\}$, $\{2, 5, 14\}$, $\{2, 9, 25\}$, $\{2, 10, 12\}$, $\{2, 11, 13\}$, $\{2, 15, 22\}$,\adfsplit
 $\{2, 16, 24\}$, $\{2, 17, 26\}$, $\{2, 18, 19\}$, $\{2, 20, 23\}$, $\{2, 21, 27\}$, $\{3, 4, 18\}$,\adfsplit
 $\{3, 5, 24\}$, $\{3, 6, 23\}$, $\{3, 7, 27\}$, $\{3, 10, 13\}$, $\{3, 11, 19\}$, $\{3, 12, 26\}$,\adfsplit
 $\{3, 14, 15\}$, $\{3, 16, 17\}$, $\{3, 20, 22\}$, $\{3, 21, 25\}$, $\{4, 6, 14\}$, $\{4, 7, 16\}$,\adfsplit
 $\{4, 9, 13\}$, $\{4, 12, 25\}$, $\{4, 15, 23\}$, $\{4, 17, 27\}$, $\{4, 19, 26\}$, $\{4, 20, 24\}$,\adfsplit
 $\{4, 21, 22\}$, $\{5, 6, 18\}$, $\{5, 7, 19\}$, $\{5, 8, 25\}$, $\{5, 9, 26\}$, $\{5, 10, 17\}$,\adfsplit
 $\{5, 11, 22\}$, $\{5, 15, 27\}$, $\{5, 16, 20\}$, $\{5, 21, 23\}$, $\{6, 8, 27\}$, $\{6, 9, 19\}$,\adfsplit
 $\{6, 11, 26\}$, $\{6, 13, 17\}$, $\{6, 16, 22\}$, $\{6, 20, 25\}$, $\{6, 21, 24\}$, $\{7, 8, 26\}$,\adfsplit
 $\{7, 9, 15\}$, $\{7, 10, 21\}$, $\{7, 12, 23\}$, $\{7, 14, 25\}$, $\{7, 17, 22\}$, $\{7, 18, 24\}$,\adfsplit
 $\{8, 10, 20\}$, $\{8, 11, 12\}$, $\{8, 13, 22\}$, $\{8, 17, 23\}$, $\{8, 18, 21\}$, $\{8, 19, 24\}$,\adfsplit
 $\{9, 10, 22\}$, $\{9, 11, 21\}$, $\{9, 12, 24\}$, $\{9, 14, 27\}$, $\{9, 18, 20\}$, $\{10, 15, 24\}$,\adfsplit
 $\{10, 16, 23\}$, $\{10, 18, 26\}$, $\{10, 19, 25\}$, $\{11, 14, 24\}$, $\{11, 17, 25\}$, $\{11, 18, 23\}$,\adfsplit
 $\{11, 20, 27\}$, $\{12, 14, 21\}$, $\{12, 15, 17\}$, $\{12, 16, 27\}$, $\{12, 19, 22\}$, $\{13, 14, 23\}$,\adfsplit
 $\{13, 15, 25\}$, $\{13, 16, 26\}$, $\{13, 18, 27\}$, $\{13, 20, 21\}$, $\{14, 16, 18\}$, $\{14, 17, 19\}$,\adfsplit
 $\{15, 16, 19\}$, $\{15, 20, 26\}$, $\{22, 23, 24\}$, $\{25, 26, 27\}$. 
}

\noindent The opposite line pair is $\{22,23,24\}$, $\{25,26,27\}$.~\\

\section*{Appendix C}\label{app:PENT4}

{\noindent\boldmath $\adfPENT(4, 21)$}.\\
With point set $Z_{68}$, the 357 lines are generated from

{\small
 $\{9, 29, 32, 36\}$, $\{3, 40, 58, 60\}$, $\{13, 19, 39, 47\}$, $\{1, 26, 34, 54\}$,\adfsplit
 $\{67, 23, 58, 8\}$, $\{58, 2, 55, 48\}$, $\{41, 8, 57, 36\}$, $\{15, 28, 11, 65\}$,\adfsplit
 $\{29, 6, 10, 5\}$, $\{0, 2, 12, 46\}$, $\{0, 26, 53, 57\}$, $\{1, 55, 66, 67\}$,\adfsplit
 $\{0, 25, 38, 54\}$, $\{0, 1, 43, 62\}$, $\{0, 17, 27, 63\}$, $\{0, 8, 24, 47\}$,\adfsplit
 $\{0, 6, 42, 67\}$, $\{0, 14, 19, 35\}$, $\{0, 3, 13, 30\}$, $\{1, 10, 39, 57\}$,\adfsplit
 $\{1, 9, 31, 37\}$
}

under the action of the mapping $x \mapsto x + 4 \adfmod{68}$.\\
The deficiency graph is connected and has girth 6.~\\

{\noindent\boldmath $\adfPENT(4, 29)$}.\\
With point set $Z_{92}$, the 667 lines are generated from

{\small
 $\{14, 16, 69, 76\}$, $\{24, 33, 61, 63\}$, $\{23, 46, 50, 80\}$, $\{33, 39, 59, 74\}$,\adfsplit
 $\{33, 69, 49, 26\}$, $\{35, 59, 77, 67\}$, $\{58, 20, 67, 42\}$, $\{0, 27, 71, 49\}$,\adfsplit
 $\{45, 57, 19, 30\}$, $\{75, 84, 47, 32\}$, $\{24, 47, 18, 10\}$, $\{66, 71, 54, 53\}$,\adfsplit
 $\{87, 26, 54, 20\}$, $\{72, 91, 60, 1\}$, $\{0, 63, 66, 79\}$, $\{0, 29, 75, 87\}$,\adfsplit
 $\{0, 13, 51, 91\}$, $\{2, 3, 22, 54\}$, $\{0, 3, 7, 50\}$, $\{1, 9, 43, 82\}$,\adfsplit
 $\{0, 11, 17, 44\}$, $\{1, 5, 15, 22\}$, $\{0, 26, 57, 82\}$, $\{0, 1, 25, 54\}$,\adfsplit
 $\{0, 42, 45, 89\}$, $\{0, 41, 46, 81\}$, $\{0, 4, 72, 77\}$, $\{0, 2, 61, 70\}$,\adfsplit
 $\{0, 8, 18, 36\}$
}

under the action of the mapping $x \mapsto x + 4 \adfmod{92}$.\\
The deficiency graph is connected and has girth 5.~\\

{

{\noindent\boldmath $\adfPENT(4, 37)$}.\\
With point set $Z_{116}$, the 1073 lines are generated from

{\small
 $\{3, 14, 63, 79\}$, $\{15, 33, 54, 85\}$, $\{10, 65, 104, 110\}$, $\{0, 40, 56, 105\}$,\adfsplit
 $\{55, 47, 38, 26\}$, $\{79, 85, 107, 26\}$, $\{87, 28, 21, 69\}$, $\{67, 100, 70, 22\}$,\adfsplit
 $\{82, 105, 20, 115\}$, $\{82, 95, 24, 23\}$, $\{85, 105, 47, 90\}$, $\{4, 52, 94, 97\}$,\adfsplit
 $\{32, 62, 26, 30\}$, $\{59, 108, 44, 73\}$, $\{0, 99, 101, 104\}$, $\{85, 46, 31, 19\}$,\adfsplit
 $\{109, 52, 102, 84\}$, $\{17, 66, 6, 93\}$, $\{86, 79, 10, 11\}$, $\{47, 57, 41, 66\}$,\adfsplit
 $\{0, 2, 21, 94\}$, $\{1, 18, 38, 82\}$, $\{0, 8, 34, 80\}$, $\{0, 13, 54, 82\}$,\adfsplit
 $\{0, 28, 61, 74\}$, $\{0, 9, 10, 35\}$, $\{0, 1, 43, 66\}$, $\{0, 23, 47, 78\}$,\adfsplit
 $\{0, 4, 73, 106\}$, $\{0, 19, 98, 103\}$, $\{0, 37, 39, 91\}$, $\{1, 5, 13, 93\}$,\adfsplit
 $\{0, 5, 24, 75\}$, $\{0, 17, 55, 89\}$, $\{0, 11, 31, 53\}$, $\{1, 31, 35, 61\}$,\adfsplit
 $\{0, 7, 87, 96\}$
}

under the action of the mapping $x \mapsto x + 4 \adfmod{116}$.\\
The deficiency graph is connected and has girth 6.~\\

{\noindent\boldmath $\adfPENT(4, 40)$}.\\
With point set $Z_{125}$, the 1250 lines are generated from

{\small
 $\{49, 97, 101, 119\}$, $\{25, 49, 56, 71\}$, $\{12, 30, 113, 117\}$, $\{17, 48, 54, 83\}$,\adfsplit
 $\{10, 78, 80, 81\}$, $\{70, 11, 47, 98\}$, $\{78, 75, 2, 71\}$, $\{3, 16, 21, 120\}$,\adfsplit
 $\{118, 74, 82, 114\}$, $\{40, 69, 13, 90\}$, $\{78, 19, 108, 61\}$, $\{25, 115, 78, 66\}$,\adfsplit
 $\{33, 55, 22, 10\}$, $\{83, 1, 28, 62\}$, $\{119, 6, 28, 43\}$, $\{114, 113, 52, 107\}$,\adfsplit
 $\{123, 7, 8, 114\}$, $\{120, 25, 15, 6\}$, $\{98, 84, 25, 12\}$, $\{104, 79, 51, 94\}$,\adfsplit
 $\{112, 95, 101, 99\}$, $\{25, 50, 88, 72\}$, $\{61, 95, 111, 27\}$, $\{77, 124, 79, 0\}$,\adfsplit
 $\{13, 32, 95, 88\}$, $\{51, 115, 74, 54\}$, $\{4, 72, 30, 35\}$, $\{46, 104, 7, 22\}$,\adfsplit
 $\{0, 14, 74, 85\}$, $\{0, 13, 44, 60\}$, $\{0, 15, 51, 108\}$, $\{0, 33, 56, 58\}$,\adfsplit
 $\{0, 19, 111, 117\}$, $\{0, 18, 72, 89\}$, $\{0, 34, 39, 69\}$, $\{0, 32, 81, 113\}$,\adfsplit
 $\{0, 2, 7, 82\}$, $\{0, 27, 67, 86\}$, $\{0, 9, 96, 122\}$, $\{1, 14, 53, 64\}$,\adfsplit
 $\{3, 23, 64, 88\}$, $\{1, 22, 63, 109\}$, $\{1, 79, 82, 112\}$, $\{1, 2, 31, 98\}$,\adfsplit
 $\{1, 9, 36, 46\}$, $\{1, 26, 72, 118\}$, $\{1, 73, 78, 94\}$, $\{2, 27, 69, 92\}$,\adfsplit
 $\{1, 21, 52, 61\}$, $\{1, 17, 43, 69\}$
}

under the action of the mapping $x \mapsto x + 5 \adfmod{125}$.\\
The deficiency graph is connected and has girth 5.~\\

{\noindent\boldmath $\adfPENT(4, 52)$}.\\
With point set $Z_{161}$, the 2093 lines are generated from

{\small
 $\{14, 76, 131, 147\}$, $\{16, 32, 44, 110\}$, $\{9, 120, 148, 156\}$, $\{83, 124, 132, 158\}$,\adfsplit
 $\{10, 34, 134, 138\}$, $\{32, 35, 45, 57\}$, $\{38, 87, 91, 137\}$, $\{80, 159, 52, 91\}$,\adfsplit
 $\{134, 158, 41, 54\}$, $\{109, 64, 57, 11\}$, $\{160, 151, 56, 29\}$, $\{89, 27, 117, 70\}$,\adfsplit
 $\{137, 115, 138, 122\}$, $\{133, 145, 27, 29\}$, $\{139, 75, 46, 3\}$, $\{157, 159, 47, 30\}$,\adfsplit
 $\{79, 149, 83, 44\}$, $\{107, 122, 48, 158\}$, $\{150, 75, 130, 6\}$, $\{29, 32, 136, 139\}$,\adfsplit
 $\{145, 80, 158, 19\}$, $\{28, 55, 27, 13\}$, $\{121, 118, 124, 130\}$, $\{114, 88, 12, 15\}$,\adfsplit
 $\{65, 55, 122, 142\}$, $\{61, 13, 139, 159\}$, $\{20, 65, 2, 106\}$, $\{78, 17, 144, 62\}$,\adfsplit
 $\{70, 150, 91, 94\}$, $\{117, 127, 64, 146\}$, $\{42, 138, 131, 23\}$, $\{145, 21, 44, 120\}$,\adfsplit
 $\{19, 141, 109, 24\}$, $\{147, 73, 23, 106\}$, $\{19, 92, 57, 87\}$, $\{154, 85, 141, 159\}$,\adfsplit
 $\{71, 135, 155, 11\}$, $\{0, 1, 61, 103\}$, $\{0, 2, 82, 152\}$, $\{0, 4, 40, 54\}$,\adfsplit
 $\{0, 16, 138, 159\}$, $\{1, 2, 33, 138\}$, $\{0, 11, 26, 110\}$, $\{0, 29, 75, 141\}$,\adfsplit
 $\{1, 3, 12, 143\}$, $\{0, 6, 33, 50\}$, $\{2, 4, 12, 62\}$, $\{1, 6, 97, 131\}$,\adfsplit
 $\{2, 3, 40, 44\}$, $\{1, 10, 24, 68\}$, $\{2, 13, 89, 115\}$, $\{2, 23, 47, 66\}$,\adfsplit
 $\{2, 54, 55, 101\}$, $\{2, 75, 87, 95\}$, $\{4, 32, 124, 146\}$, $\{4, 25, 68, 116\}$,\adfsplit
 $\{0, 7, 15, 85\}$, $\{0, 17, 113, 127\}$, $\{0, 9, 43, 115\}$, $\{1, 11, 22, 34\}$,\adfsplit
 $\{1, 17, 27, 115\}$, $\{0, 25, 97, 106\}$, $\{0, 13, 31, 71\}$, $\{0, 36, 65, 155\}$,\adfsplit
 $\{1, 80, 144, 149\}$, $\{1, 48, 55, 88\}$, $\{1, 60, 95, 135\}$, $\{1, 90, 94, 107\}$,\adfsplit
 $\{1, 59, 125, 136\}$, $\{1, 39, 41, 156\}$, $\{0, 60, 64, 156\}$, $\{1, 87, 114, 139\}$,\adfsplit
 $\{2, 10, 116, 132\}$, $\{0, 38, 73, 98\}$, $\{2, 31, 139, 144\}$, $\{2, 38, 80, 109\}$,\adfsplit
 $\{3, 4, 73, 104\}$, $\{0, 39, 66, 72\}$, $\{0, 30, 52, 84\}$, $\{0, 49, 102, 143\}$,\adfsplit
 $\{0, 45, 83, 135\}$, $\{0, 48, 69, 132\}$, $\{0, 20, 93, 125\}$, $\{0, 34, 56, 114\}$,\adfsplit
 $\{0, 42, 86, 153\}$, $\{0, 41, 79, 137\}$, $\{0, 18, 100, 149\}$, $\{0, 67, 128, 151\}$,\adfsplit
 $\{0, 32, 70, 121\}$, $\{0, 35, 109, 116\}$, $\{0, 88, 130, 144\}$
}

under the action of the mapping $x \mapsto x + 7 \adfmod{161}$.\\
The deficiency graph is connected and has girth 5.~\\

{\noindent\boldmath $\adfPENT(4, 60)$}.\\
With point set $Z_{185}$, the 2775 lines are generated from

{\small
 $\{44, 56, 75, 110\}$, $\{3, 71, 116, 130\}$, $\{12, 114, 149, 177\}$, $\{1, 24, 48, 143\}$,\adfsplit
 $\{42, 77, 145, 168\}$, $\{52, 74, 76, 149\}$, $\{76, 7, 34, 32\}$, $\{13, 142, 178, 45\}$,\adfsplit
 $\{132, 28, 162, 109\}$, $\{3, 45, 145, 151\}$, $\{137, 59, 77, 130\}$, $\{103, 177, 20, 62\}$,\adfsplit
 $\{98, 123, 9, 134\}$, $\{44, 6, 11, 98\}$, $\{119, 139, 24, 89\}$, $\{176, 15, 16, 93\}$,\adfsplit
 $\{126, 154, 24, 37\}$, $\{46, 87, 121, 99\}$, $\{128, 140, 132, 10\}$, $\{144, 15, 8, 46\}$,\adfsplit
 $\{135, 14, 162, 96\}$, $\{154, 74, 78, 150\}$, $\{76, 48, 29, 19\}$, $\{41, 70, 72, 83\}$,\adfsplit
 $\{28, 116, 178, 12\}$, $\{12, 4, 21, 88\}$, $\{87, 91, 118, 12\}$, $\{55, 85, 52, 109\}$,\adfsplit
 $\{122, 121, 71, 129\}$, $\{76, 23, 145, 85\}$, $\{158, 144, 88, 153\}$, $\{147, 144, 21, 67\}$,\adfsplit
 $\{75, 155, 32, 173\}$, $\{161, 171, 12, 125\}$, $\{57, 33, 74, 106\}$, $\{99, 178, 166, 150\}$,\adfsplit
 $\{137, 90, 176, 52\}$, $\{1, 65, 114, 153\}$, $\{68, 153, 82, 146\}$, $\{119, 34, 28, 160\}$,\adfsplit
 $\{37, 120, 115, 172\}$, $\{103, 53, 70, 21\}$, $\{169, 123, 24, 30\}$, $\{99, 51, 160, 86\}$,\adfsplit
 $\{169, 127, 28, 153\}$, $\{88, 27, 73, 122\}$, $\{158, 90, 128, 4\}$, $\{183, 21, 24, 177\}$,\adfsplit
 $\{0, 37, 103, 183\}$, $\{0, 8, 9, 120\}$, $\{0, 108, 163, 172\}$, $\{0, 48, 137, 158\}$,\adfsplit
 $\{1, 19, 53, 56\}$, $\{1, 23, 31, 168\}$, $\{3, 13, 124, 129\}$, $\{0, 14, 21, 148\}$,\adfsplit
 $\{0, 3, 66, 87\}$, $\{0, 81, 138, 167\}$, $\{1, 17, 18, 72\}$, $\{0, 12, 128, 141\}$,\adfsplit
 $\{1, 44, 102, 167\}$, $\{1, 16, 92, 106\}$, $\{2, 7, 54, 147\}$, $\{1, 66, 134, 159\}$,\adfsplit
 $\{0, 84, 127, 181\}$, $\{0, 41, 61, 101\}$, $\{0, 19, 96, 159\}$, $\{0, 51, 151, 162\}$,\adfsplit
 $\{0, 22, 89, 104\}$, $\{0, 77, 92, 164\}$, $\{0, 11, 17, 170\}$, $\{0, 20, 91, 169\}$,\adfsplit
 $\{0, 25, 94, 95\}$, $\{0, 34, 67, 140\}$, $\{0, 10, 39, 50\}$
}

under the action of the mapping $x \mapsto x + 5 \adfmod{185}$.\\
The deficiency graph is connected and has girth 5.~\\

\end{document}